\documentclass[11pt,a4paper]{article}
\date{}
\usepackage{enumerate,amsmath,amsthm,amssymb,latexsym,
amsfonts,amscd}
\usepackage[all]{xy}

\title{Differential algebras with Banach-algebra
coefficients I: From C*-algebras to the K-theory of the spectral
curve}

\author{
Maurice J. Dupr\'e \\ Department of Mathematics, Tulane
University\\ New Orleans, LA 70118 USA\\ mdupre@tulane.edu \\ \\
James  F. Glazebrook
\\(Primary Inst.)\\ Department of Mathematics and Computer
Science \\
 Eastern Illinois University \\
600  Lincoln Ave., Charleston, IL 61920--3099 USA \\
jfglazebrook@eiu.edu
\\ (Adjunct Faculty)
\\ Department of Mathematics \\ University of Illinois at
Urbana--Champaign\\ Urbana, IL 61801, USA
\\
\\ Emma Previato\thanks{Partial research support under grant
NSF-DMS-0808708 is very gratefully acknowledged.}
\\ Department of Mathematics and Statistics, Boston University
\\Boston, MA 02215--2411, USA \\
ep@math.bu.edu }

\theoremstyle{plain}
\newtheorem{lemma}{Lemma}[section]
\newtheorem{proposition}{Proposition}[section]
\newtheorem{theorem}{Theorem}[section]

\theoremstyle{definition}
\newtheorem{definition}{Definition}[section]
\newtheorem{example}{Example}[section]
\newtheorem{remark}{Remark}[section]

\numberwithin{equation}{section}

\newcommand{\alg}{{\rm alg}}

\newcommand{\Det}{{\rm Det}}

\newcommand{\Sim}{{\rm Sim}}

\newcommand{\ch}{{\rm ch}}

\newcommand{\Ext}{{\rm Ext}}
\newcommand{\Fred}{{\rm Fred}}

\newcommand{\Gr}{{\rm Gr}}

\newcommand{\Hom}{{\rm Hom}}

\newcommand{\IM}{{\rm Im}}
\newcommand{\Int}{{\rm Int}}

\newcommand{\Mero}{{\mathsf{Mero}}}

\newcommand{\sa}{{\rm sa}}

\newcommand{\Spec}{{\rm Spec}}

\newcommand{\Tr}{{\rm Tr}}

\newcommand{\vp}{\varphi}

\newcommand{\A}{\boldsymbol{\mathcal{A}}}
\newcommand{\B}{\mathcal B}

\newcommand{\E}{\mathcal E}

\newcommand{\K}{\mathcal K}

\newcommand{\cL}{\mathcal L}

\newcommand{\Q}{\mathcal Q}

\newcommand{\T}{\mathcal T}

\newcommand{\bA}{\mathbb{A}}
\newcommand{\bB}{\mathbb{B}}
\newcommand{\bC}{\mathbb{C}}

\newcommand{\bQ}{\mathbb{Q}}

\newcommand{\bZ}{\mathbb{Z}}

\newcommand{\hp}{\mathsf{H}_{+}}
\newcommand{\hm}{\mathsf{H}_{-}}

\newcommand{\hpm}{\mathsf{H}_{\pm}}

\newcommand{\esf}{\mathsf{e}}

\newcommand{\ra}{\rightarrow}

\newcommand{\lra}{\longrightarrow}

\newcommand{\ovsetl}[1]{\overset {#1}{\lra}}

\newcommand{\ul}{\underline}
\newcommand{\wti}{\widetilde}

\newcommand{\uA}{\underline{\mathbb{A}}}
\newcommand{\uB}{\underline{\mathbb{B}}}
\newcommand{\Acl}{\bar{{A}}}

\newcommand{\sfT}{\mathsf{T}}

\newcommand{\del}{\partial}
\newcommand{\delbar}{\bar{\partial}}

\newcommand{\dm}{\del^{-1}}

\newcommand{\med}{\medbreak}

\begin{document}

\maketitle

\begin{abstract}
We present an operator-coefficient version
of Sato's infinite-dimensional Grassmann manifold, and $\tau$-function.
In this context, the Burchnall-Chaundy ring of commuting
differential operators becomes a C*-algebra, to which we apply the
Brown-Douglas-Fillmore theory, and
topological invariants of the spectral ring become readily available.
We construct KK classes of the spectral curve
of the ring and, motivated by the fact that all isospectral
Burchnall-Chaundy rings make up the Jacobian of the curve, we
compare the (degree-1) K-homology of the curve with that of its
Jacobian. We show how the Burchnall-Chaundy C*-algebra
extension of the compact
operators provides a family of operator-valued $\tau$-functions.
\end{abstract}

\med \textbf{Mathematics Subject Classification (2010)} Primary:
46L08, 19K33,
 19K35.
Secondary: 13N10, 14H40, 47C15

\med \textbf{Keywords}: C*-algebras, extension groups of operator
algebras,
 K-homology,
spectral curve, Baker function, $\tau$-function.

\section{Introduction}\label{intro}

This work arises from \cite{DGP1,DGP3}
where we started an operator-theoretic, Banach algebra approach
to the Sato-Segal-Wilson theory in the setting of Hilbert modules
with the extension of the classical Baker and
$\tau$-functions to operator-valued functions.
Briefly recalled, the Sato-Segal-Wilson theory \cite{Sato,SW} in the context of
integrable Partial Differential Equations\footnote{The KP hierarchy
is the model we use; variations would utilize slightly different
objects of abstract algebra, e.g. matrices instead of scalars.}  (PDEs)
produces a linearization of the time flows, as one-parameter
groups acting on Sato's `Universal Grassmann Manifold',
via conjugation of the trivial solution ($\tau \equiv 0$)
by a formal pseudodifferential operator.
An important issue, which underlies our work, is
that the Grassmannian considered is analytic as opposed to formal.
The algebro-geometric solutions, which we focus on, come
from a Riemann surface, or from an algebraic curve (the former
case is our default assumption, that is to say,
the curve is smooth) which is a `spectral curve' in the sense of
J. L. Burchnall
and T. W. Chaundy (see e.g. \cite{Krich1} for an updated account
of their work). These authors in the 1920's gave  a classification of the
rank-one\footnote{Namely, such that the orders of the operators in
the ring are not all divisible by some number $r>1$; if they are,
the ring corresponds to a rank-$r$ vector bundle.} commutative rings
of Ordinary Differential Operators (ODOs) with  fixed spectral
curve, as the Jacobi variety of the curve (line bundles of fixed
degree, up to isomorphism),
using differential algebra.
 The Krichever map \cite{Krich1,SW} sends a
quintuple of holomorphic data associated to the ring to Sato's
infinite-dimensional Grassmann manifold.

Using the Sato correspondence we used a
conjugating action by an integral operator and showed
 how the conjugated Burchnall-Chaundy ring $\bA$ of
pseudodifferential operators can be represented as a commutative
subring of a certain Banach *-algebra $A$, a `restricted'
subalgebra of the bounded linear operators $\cL(H_{\A})$ where
$H_{\A}$ is a \emph{Hilbert module} over a C*-algebra $\A$.
A particular feature in this work is the Banach
Grassmannian $\Gr(p,A)$ which resembles in a
more general sense the restricted Grassmannians formerly introduced
in \cite{PS,SW} and these latter spaces can be recovered as
subspaces of $\Gr(p,A)$. Some of our results require the algebra
$\A$ to be commutative and therefore $\A \cong
C(Y)$, for some compact Hausdorff space $Y$.
The commutative ring $\bA$ has as its joint spectrum $X'=
\Spec(\bA)$ an irreducible complex curve whose one-point
compactification is a non-singular (by assumption)
algebraic curve $X$ of genus $g_X
\geq 1$. With $\A$ being commutative, we obtain a $Y$-parametrized
version of the Krichever correspondence (originally
given for the case $\A = \bC$).

To transfer information between the geometry of integrable
PDEs and that of operator algebras, we
show that naturally associated to the
ring $\bA$ is a C*-algebra $\uA$ (which we refer to as the
\emph{Burchnall-Chaundy C*-algebra}) which, by the
Gelfand-Neumark-Segal theorem, can be realized as a C*-algebra of
operators on an associated Hilbert space $H(\A)$. For any generating
$s$-tuple of commuting operators in $\uA$ we consider its joint
spectrum and show that it is homeomorphic to the joint spectrum
$X'$.

This paper also considers the K-homology of $C(X)$ and
related theory. In particular, by the Brown-Douglas-Fillmore (BDF)
theorem, the group $K^1(C(X))$ is the extension group $\Ext(X)$, of
the compact operators $\K(H)$ by $C(X)$, in the sense that there is
a natural transformation of covariant functors
$$ \Ext(X) \lra \Hom (K^1(X), \bZ), $$ and this extension group,
in turn, can be identified with the C*-isomorphisms of $C(X)$ into
the Calkin algebra $\Q(H)$. Using only the commutative subalgebra we
constructed from the Burchnall-Chaundy ring, for the case
$\A=\mathbb{C}$, we obtain some of these C*-isomorphisms,
parametrizing in fact the Jacobian $J(X)$ of $X$ (which is the group
$\Ext^1(\mathcal O_X, \mathcal O_X^*)\cong H^1(X, \mathcal O^*)$).
Finally, using the Abel map which embeds $X \lra J(X)$, we fit the
three C*-algebras $\uA, C(X)$ and $C(J(X))$ into the framework of
BDF theory \cite{BDF}. We show, using the basic instrumentation of K-homology, that there exists
a natural injection of Ext groups $\Ext(X) \lra \Ext(J(X))$ (an
isomorphism when the genus $g_X$ is 1). Finally, in this Part I we consider
the operator ($\A$-) valued $\tau$-function as defined on a group of multiplication
operators and show that extensions of the compact operators
by the Burchnall-Chaundy C*-algebra $\uA$,
yield a family of $\tau$-functions attached  to `transverse'
subspaces $W \in \Gr(p,A)$, one for each such extension.

Part II \cite{DGP4} of this work involves a number of extensions of Part I, in particular a geometric setting for
predeterminants of the $\tau$-function, called {$\T$-functions}. In this setting it is possible
pull back the tautological bundle over
the Grassmannian, and obtain, much as in the algebraic case of
\cite{Alvarez}, a Poincar\'e bundle over a product of homogeneous
varieties for operator group flows, whose fibres will be related by
the BDF theorem to the Jacobians of the spectral curves. An operator
cross-ratio on the fibres provides the $\tau$-function, and
we  investigate the Schwarzian
derivative associated to the cross-ratio.
As a result, operator-valued objects defined
with a minimum of topological assumptions  obey the same
hierarchies as the corresponding objects arising from
the theory of special functions.
Then we bring together the main results of this Part I to show in Part II \cite{DGP4}
that solutions to the KP hierarchy can in fact be parametrized by elements of
$\Ext(\uA)$. Thus we view our results as an enhancement for operator theory
and a potential enhancement for the
theory under development of geometric quantization and integrable PDEs,
where functional Hamiltonians are replaced by operators. Other approaches
in the operator-theoretic setting as applied to control theory
are studied in \cite{Helton} and to Lax-Phillips scattering in \cite{Ball2}.
The present work could in part be viewed as extensions of the Cowen-Douglas theory
\cite{CD,MS2} linking complex analytic/algebraic techniques to operator theory,
thus suggesting future developments.

This Part I is organized as follows: for the reader's sake we recall the necessary
setting and details from \cite{DGP1} and other relevant works in an Appendix \S\ref{banachable}
including our parametrized version of the Krichever
correspondence,
and proceed from \S\ref{basic} to the construction of the above-described
C*-algebras in \S\ref{C*algebra}.  In \S\ref{extension} we
pursue the BDF theory and attach topological
objects to the spectral curve and Burchnall-Chaundy ring.

The new results we have obtained in this Part I are Theorem \ref{spectra}, Propositions \ref{family}, \ref{jacobian},
Theorems \ref{SES}, \ref{map}, \ref{commutative}, \ref{tau-ext-theorem}.


\section{The basic setting}\label{basic}

\subsection{The restricted Banach *-algebra $A$ and the space
$\Gr(p,A)$}\label{restricted-1}

Let $A$ be a unital complex
Banach(able) algebra with group of units $G(A)$ and space of
idempotents $P(A)$. An important carrier for what follows
is an infinite dimensional Grassmannian, denoted $\Gr(p,A)$, whose structure as a Banach homogeneous space
is briefly described in Appendix \S\ref{projections} as well as its universal bundle.
Similar spaces and related objects have been used by a number of authors in dealing with a variety of
different operator-geometric problems (see \S\ref{projections} Remark \ref{others-remark}).

\begin{remark}\label{proj-remark}
We point out that we use the term projection to simply mean an
idempotent even though functional analysts often use the term
`projection' to mean a self-adjoint idempotent when dealing with
a *-algebra.  Consequently we will explicitly state when an idempotent
is or should be required to be self-adjoint. In the latter case we
denote the set of these as $P^{\sa}(A)$.
\end{remark}

Let $H$ be a separable (infinite dimensional) Hilbert space.
Given a unital separable C*-algebra $\A$, we take the
standard (free countable dimensional) Hilbert module $H_{\A}$ over
$\A$ (see Appendix \S\ref{semigroup})
and consider a \emph{polarization} of $H_{\A}$ given by a pair of submodules
$(\hp, \hm)$, such that
\begin{equation}\label{polar1}
H_{\A} = \hp \oplus \hm ~,~\text{and}~ \hp \cap \hm = \{0\}.
\end{equation}

If we have a unitary $\A$--module map $J$ satisfying
$J^2 = 1$, there is an induced eigenspace decomposition $H_{\A} =
\hp \oplus \hm$, for which $\hpm \cong H_{\A}$. This leads
to the (restricted) Banach algebra $A = \cL_J(H_{\A})$ as described in
\cite{DGP1} (generalizing that of $\A = \bC$ in \cite{PS}).
Specifically, we have
\begin{equation}\label{staralgebra}
A = \cL_J(H_{\A}) := \{\sfT \in \cL_{\A} (H_{\A}) : \sfT ~\text{is
adjointable and}~ [J,\sfT] \in \cL_2(H_{\A}) \} ,
\end{equation}
(by definition $\cL_2(H_{\A})$ denotes the Hilbert-Schmidt
operators) where for $\sfT \in A$, we assign the norm (cf.
\cite[\S6.2]{PS}):
\begin{equation}\label{jnorm}
\Vert \sfT \Vert_J = \Vert \sfT \Vert ~+~ \Vert [J,\sfT]\Vert_{2}.
\end{equation}
The algebra $A$ can thus be seen as a Banach *-algebra with
isometric involution (when $\A \cong \bC$ we simply write
$\cL_J(H)$). This is our `restricted' algebra which we will use
henceforth. Together with the topology induced by $\Vert~ \Vert_J$,
the group of units $G(A)$ is a complex Banach Lie group for which we
have the unitary Lie subgroup $U(A) \subset G(A)$. Further, $\Gr(p,A)$
as a Banach manifold, can also be equipped with a Hermitian structure
(see \S \ref{projections} and references therein).

We denote by $\Acl$, the C*-algebra norm closure of
$A$ in the C*-algebra $\cL(H_{\A})$. By standard principles, $A$ is thus a
Banach
*-algebra with isometric involution dense in the
C*-algebra $\Acl$. In particular, as $A$ is a Banach algebra, it is
closed under the holomorphic functional calculus, and as it is a
*-subalgebra of the C*-algebra $\Acl$, it is therefore a (unital)
\emph{pre-C*-algebra} in the sense of \cite{Bl,Connesbook}.


\begin{remark}
For the most part of this work we shall be taking the Hilbert space to be $H = L^2(S^{1}, \bC)$
without loss of generality.
\end{remark}

\subsection{The case where $\A$ is commutative}\label{commcase}

We shall henceforth assume that
$\A$ is a commutative separable
C*-algebra. The Gelfand transform implies there exists a compact
metric space $Y$ such that $Y = \Spec(\A)$ and $\A \cong C(Y)$.
Setting $B = \cL_J(H)$, we can now express the Banach *-algebra
$A$ in the form
\begin{equation}\label{aform}
A \cong \{\text{continuous functions}~ Y \lra B
\} = C(Y,B),
\end{equation}
for which the $\Vert ~ \Vert_2$-trace in the norm of $A$ is
regarded as continuous as a function on $Y$. Note the Banach
algebra $B = \cL_J(H)$ corresponds to taking $\A= \bC$, and with
respect to the polarization $H = H_{+} \oplus H_{-}$, we recover
the usual restricted Grassmannians, denoted $\Gr(H_{+}, H)$ of
\cite{PS,SW}.


\section{The Burchnall-Chaundy C*-algebra}\label{C*algebra}

Appendices \S\ref{essential} and \S\ref{BCring+} reviews the Burchnall-Chaundy algebra and related geometry.
We henceforth use the notation introduced there, referencing to it
as we go along.

Motivated by Theorem \ref{rep1} we obtain a natural
C*-algebra associated to the ring $\bA$ (\S\ref{essential}).
Firstly, in view of Theorem
\ref{rep1}, let
\begin{equation}\label{inclusion1}
i_K : \bA \longrightarrow A,
\end{equation}
be the inclusion map induced via conjugation by the integral
operator $K$ \eqref{integral-op} in Appendix \eqref{integralop}.
We also  recall the C*-algebra
$\Acl$ of \S\ref{restricted-1}.
\begin{definition}\label{bcalgebra}
The \emph{Burchnall-Chaundy C* algebra}, denoted $\uA$, is the
C*-subalgebra of $\Acl$ generated by $i_K(\bA)$.
\end{definition}
Thus $\uA$ is a separable C*-subalgebra of $\Acl$ (this uses the
fact that $\Acl$ is separable). We also endow $\uA$ with an identity
as induced from that of $\Acl$. Further, let
\begin{equation}\label{inclusion2}
\ul{i}_K : \bA  \lra \uA,
\end{equation}
be the induced map of \eqref{inclusion1} into the C*-algebra $\uA$
where both maps in \eqref{inclusion1}, \eqref{inclusion2} factor
through the tensor product $\bA \otimes \A$.

\begin{remark}
Note that the relevant algebras considered
above, are already inside the C*-algebra $\cL_{\A}(H_{\A})$. In view
of Remark \ref{hilbertremark}, applying a completely positive state
to pass to $H(\A)$ serves to make $\cL_{\A}(H_{\A})$ a C*-subalgebra
of $\cL(H(\A))$. Thus we may, if we wish, identify $\uA$ as a
particular C*-subalgebra of $\cL(H(\A))$.
\end{remark}

In the following we will also consider a commutative Banach
subalgebra $\uB$ of $\uA$ as generated by elements of $i_K(\bA)$
that closes the image of $i_K$ in the norm topology. As maximal
ideals are proper closed ideals, the image of $i_K$ and $\uB$ have
the same spectrum in the weak *-topology (recall that as the hull
kernel topology is completely regular and Hausdorff, it follows that
the weak *-topology and hull kernel topology are the same
\cite{Rickart}). We shall use this fact in establishing Theorem
\ref{spectra} below.

Suppose we have some finite number $s$ of commuting operators $L_i
\in \bA$, for $1 \leq i \leq s$. Recalling \eqref{inclusion2} and
setting $T_i = \ul{i}_K(L_i) \in \uB$, leads to
\begin{equation}\label{generators}
0 = \ul{i}_K([L_i, L_j]_{\bA}) = [\ul{i}_K(L_i),
\ul{i}_K(L_j)]_{\uA} = [T_i, T_j]_{\uA} ,
\end{equation}
and therefore to a commuting $s$-tuple $\T(s) \equiv (T_1,
\ldots, T_s) \in \uB^s$.

We next specify the connection between the joint spectrum of the
Burchnall-Chaundy ring $\bA$ and the joint spectrum $\sigma(\T(s),
\uB)$ of  a generating commuting $s$-tuple $\T(s)$ of operators in
the Banach subalgebra $\uB$.
\begin{theorem}\label{spectra}
For any $s$-tuple $\T(s)\in \uB^s$ of commuting operators that
generate $\uB$, the map of spectra
\begin{equation}
\sigma(\T(s), \uB) \lra  (X' = \Spec(\bA)) \times Y,
\end{equation}
is a homeomorphism.
\end{theorem}
\begin{proof}
Initially, it is enough to deal with the case $Y=\{\rm{pt}\}$.
Recall that each $T_i = \ul{i}_K(L_i) \in \uB$, for $1 \leq i \leq
s$. By definition of the spectrum, any point of the latter can be
regarded as a non--trivial algebra homomorphism $\uB \ovsetl{f}
\bC$, which by definition restricts to the ring of generators $\bA$,
that is, we have a restricted homomorphism $f\vert \bA: \bA
\ovsetl{f} \bC$. Note that if $f_1,f_2$ are two such homomorphisms,
then if $f_1=f_2$ on $\bA$, we have $f_1=f_2$ on $\ul{i}_K(\bA)$.

Now given a complex homomorphism on the image of $\ul{i}_K$, which
is a normed continuous ring dense in $\uB$, then viewing this as a
linear functional, it admits a closed kernel which splits the normed
ring as a topological vector space, and consequently the closure of
this kernel in $\uB$ will be a closed maximal ideal of $\uB$ which
therefore defines a complex homomorphism extending that from the
image of $\ul{i}_K$. We have then $f_1=f_2$ on $\uB$. Thus the
restriction map is a continuous bijection of the spectrum
$\sigma(\T(s), \uB)$ onto the spectrum of $\ul{i}_K$, and hence that
of $\bA$. As both spaces are compact Hausdorff spaces, this map is a
homeomorphism. The more general conclusion follows on parametrizing
by $Y$.
\end{proof}

\begin{remark} The $s$-tuple $\T(s) = (T_1, \ldots, T_s)$ of
commuting operators in $\uB$, being images $T_i = \ul{i}_K(L_i)$ of
operators that commute with $L$ of order $n$, provide a solution to
the $n$-th generalized KdV-hierarchy, simply
reproducing the construction in \cite[Proposition 4.12, Corollary
5.18, and \S6]{SW}.
\end{remark}


\section{Extension of compact operators and K--homology}
\label{extension}

\subsection{Ext and KK-classes}\label{extension1}

Let us now return to the (non-singular) algebraic curve $X = X' \cup
\{x_{\infty}\}$, of genus $g_X$, associated to $\Spec(\bA)$.
Initially, we shall consider a fixed $y \in Y$ as in the case $\A
\cong \bC$, and in that case, replace $H_{\A}$ by $H = L^2(S^1,
\bC)$. Application of the BDF extension theory
\cite{BDF} shows that an extension of the compact operators $\K(H)$
by $C(X)$ yields a unital *-monomorphism to the Calkin algebra
$\Q(H)$:
\begin{equation}\label{bdf0}
\varrho : C(X) \lra \Q(H): = \cL(H)/\K(H).
\end{equation}
The group $\Ext(X)$ of these extensions is the same as the
degree-$1$ K-homology group $K^1(C(X))$ \cite{Bl,BDF}. More
generally, $K^*(C(X))$ can be identified with Kasparov's KK-group
$KK_{*}(C(X), \bC)$; we refer the reader to e.g.
\cite{Bl,Connesbook,Kasparov,Sk} for details.

The relevance of this last observation is as follows. We recall that
in \S\ref{essential} there are subspaces $W \in \Gr(q,B)$ which are
images of the Krichever correspondence (these are characterized by
the size of the ring $B_W$ in \eqref{bwalgebra}; see
\cite[Remark 6.3]{SW}). One datum of the
Krichever correspondence was a holomorphic line bundle $\cL \lra X$.
Restricting to rank-1 projections we replace $\Gr(q,B)$ by the (infinite dimensional)
projective space $\mathbb P(H_{\A})$, which can be given a Hermitian
structure together with its universal (Hermitian)
line bundle $\cL(p,A) \lra \mathbb
P(H_{\A})$ (see Appendix \S\ref{projections}). The projective space $\mathbb P(H_{\A})$
is a classifying space, so  that given a suitable holomorphic map $f
: X \lra \mathbb P(H_{\A})$, the holomorphic line bundle $\cL \cong
f^*\cL(p,A)$ admits a Hermitian structure induced by this pullback.

On such a Hermitian bundle we endow a Hermitian connection $\nabla$
(see e.g. \cite[III Theorem 1.2]{Wells}), so we now have a Hermitian
holomorphic line bundle with connection $(\cL, \nabla) \lra X$. In
particular, the $(0,1)$-component $\nabla^{''}$ of $\nabla$, is
taken to be the $\delbar$--operator on sections (see e.g. \cite[III
Theorem 2.1]{Wells}). Since $X$, regarded as a compact Riemann
surface, admits the K\"ahler property, there is a Dirac operator $D
= \sqrt{2}(\delbar + \delbar^*)$ on sections (see e.g.
\cite[(2.20)]{Roe}). Thus we are in the setting of elliptic
operators on sections of vector bundles over a compact manifold from
which KK-classes can be constructed by standard procedures
\cite{Bl,CS,Sk}.

In \S\ref{ydata}, we discussed families of $Y$-parametrized line
bundles. Now instead we fix the bundle and parametrize possible
Hermitian connections $\nabla$ on $\cL \lra X$, by $Y$. That is, we
have a family of connections $\{\nabla_y\}_{y \in Y}$, and hence a
$Y$-parametrized family of Dirac operators $\{D_y\}_{y \in Y}$. Thus
by \cite{CS,Sk} such data leads to constructing elements $u$ of the
the group $KK_{*}(C(X),C(Y)) = KK_{*}(C(X),\A)$. We summarize
matters as follows.
\begin{proposition}\label{family}
The data $\{(\cL, \nabla_y)\}_{y \in Y}$ of a Hermitian holomorphic
line bundle on $X$ with Hermitian connections parametrized by $Y$,
determines a class $u \in KK_{*}(C(X), C(Y))$ parametrized by
$\mathbb P(H_{\A}) \times Y$.
\end{proposition}
In particular, the degree-$1$ component $u^{[1]}$ provides an
element of the group $\Ext(C(X), \A) \cong KK_1(C(X), \A)$, and in
the special case $\A = \bC$, an element of $\Ext(X) = K^1(C(X))$.

\begin{remark} The Krichever quintuple yields the data
of Proposition \ref{family} with $Y$ reduced to one point. In
particular, we can associate to any point of the Grassmannian in the
image of the Krichever map and to a given Hermitian structure over
the universal bundle a class $u$.
\end{remark}

\subsection{Extensions by the Burchnall-Chaundy C*
algebra}\label{extension2}

We recall the C*-Burchnall-Chaundy algebra $\uA$ and consider a
more general situation than \S\ref{extension1} using the
Hilbert module $H_{\A}$ and the ring $\bA$.
Concerning the spectrum $\Spec(\Q(H))$ of the Calkin algebra, we
refer to \cite{Dix} for the study of spectra of (quotient)
C*-algebras.

\begin{theorem} ~
\begin{itemize}
\item[(1)] With respect to the subring inclusion \eqref{inclusion2}
induced by the integral operator $K$ (see \eqref{integralop}),
there exists an extension of the compact operators $\K(H_{\A})$ by the
Burchnall-Chaundy C*-algebra $\uA$. Moreover, there also exists a
well-defined map $\ul{\varrho}_K : \bA \lra \Q(H_{\A})$.

\item[(2)] Recalling $X'= \Spec(\bA)$, there exists a well-defined map
of spectra
\begin{equation}
\Spec(\Q(H_{\A})) \cong \Spec(\Q(H)) \times Y \lra X' \times Y.
\end{equation}
\end{itemize}
\end{theorem}

\begin{proof}
Firstly, in relationship to $H_{\A}$, the Calkin C*-algebra
$\Q(H_{\A})$ can be expressed as follows:
\begin{equation}\label{calkin1}
\begin{aligned}
\Q(H_{\A}) &= \cL(H_{\A})/\K(H_{\A})\\
         &\cong \cL(H \otimes \A)/\K(H \otimes \A)\\
        &\cong (\cL(H)\otimes \A)/(\K(H) \otimes \A)\\
        & \cong (\cL(H)/\K(H))\otimes \A \\
        &= \Q(H) \otimes \A.
\end{aligned}
\end{equation}
Again, an application of the BDF extension theory
\cite{BDF} shows that an extension of the compact operators
$\K(H_{\A})$ by $\uA$, yields a unital *-monomorphism
\begin{equation}\label{bdf1}
\ul{\varrho} : \uA \lra \Q(H_{\A}).
\end{equation}
We also recall that a Fredholm module $(H_{\A}, S)$ over $\uA$
(where $S$ an essentially unitary operator)
determines an element in K--homology $K^*(\uA)$ \cite{Bl,BDF}.
 Combining the morphism $\ul{\varrho}$ in
\eqref{bdf1} with the inclusion \eqref{inclusion2} induced by the
integral operator $K$ in \S\ref{integral-op}, we produce a well defined
 map from the ring
$\bA$ to $\Q(H) \otimes \A$ viewed as the composition
\begin{equation}\label{bdf2}
\ul{\varrho}_K : \bA \lra \bA \otimes \A \ovsetl{\ul{i}_K} \uA
\ovsetl{\ul{\varrho}}
 \Q(H_{\A}).
\end{equation}
Furthermore, using \eqref{calkin1},
\eqref{bdf2} and the contravariance of the `Spec' functor, we have
a well-defined map
\begin{equation}
\begin{aligned}
\varrho_K^* &: \Spec(\Q(H_{\A})) \cong \Spec(\Q(H) \otimes \A)
\\ &\cong \Spec(\Q(H)) \times Y \lra (X' = \Spec(\bA)) \times Y.
\end{aligned}
\end{equation}
\end{proof}

\subsection{The C*-algebra of the Jacobian and extensions}

We return now to the Jacobian torus $J(X)$ of $X$ and recall that
there exists a holomorphic embedding (the Abel map) $\mu: X \lra
J(X)$ (see e.g.\cite{GH}). The following theorem establishes a
relationship between the respective commutative C*-algebras of
continuous functions $C(J(X))$ and $C(X)$ and the dual
$K_{*}$-functor:

\begin{theorem}\label{SES}
There exists a short exact sequence of C*-algebras
\begin{equation}\label{exact1}
0 \ra \mathfrak{I} \ovsetl{i} C(J(X)) \ovsetl{p} C(X) \ra 0,
\end{equation}
where $\mathfrak{I}$ is a two-sided ideal in $C(J(X))$, $i$ is
injective and $p$ is surjective. Furthermore, \eqref{exact1} induces
a periodic sequence of $K_{*}$--groups:
\begin{equation}\label{period1}
\xymatrix@C=4pc{K_0(\mathfrak{I})\ar[r]^{i_{*}} & K_{0}(C(J(X)))
\ar[r]^{p_{*}} & K_0(C(X)) \ar[d]\\ K_1(C(X)) \ar[u] & K_1(C(J(X)))
\ar[l]_{p_{*}} & K_1(\mathfrak{I})\ar[l]_{i_{*}} }
\end{equation}
\end{theorem}
\begin{proof}
Firstly, from the embedding $\mu: X \lra J(X)$ we identify $X$ as a
closed subset of $J(X)$ and by standard principles, the induced map
on continuous functions $p : C(J(X)) \lra C(X)$ is surjective.
Alternatively, we note that the transcendence degrees of the rings
of meromorphic functions $\Mero(J(X))$ and $\Mero(X)$ are
 $g_X \geq 1$ and $1$, respectively. Thus with respect to a
two-sided ideal $\mathfrak{I}_0$ in $\Mero(J(X))$ we have a short
exact sequence:
\begin{equation}
0 \ra \mathfrak{I}_0 \ovsetl{i} \Mero(J(X)) \ovsetl{p} \Mero(X) \ra
0 ,
\end{equation}
where in each case the elements of the ring can be approximated by
Laurent polynomials extendable to the continuous functions. Hence
\eqref{exact1} follows. The last part follows from the periodicity
of the $K_{*}$-functor \cite{Connesbook,Fill}.
\end{proof}

The periodicity sequence in \eqref{period1} corresponds to the
analogous sequence in the K-theory of spaces:
\begin{equation}
\xymatrix@C=4pc{K^0(J(X)/ \IM~\mu) \ar[r] & K^0(J(X)) \ar[r] &
K^0(X) \ar[d]\\ K^1(X)\ar[u] & K^1(J(X)) \ar[l] & K^1(J(X)/
\IM~\mu)\ar[l] }
\end{equation}

Now we return to K-homology and in particular, in degree $1$ the
following establishes a relationship between the `Ext' classes of
$X$ and $J(X)$ on introducing the homology Chern character
\cite{Baum,Connesbook}. Note that this is the case in which the BDF
theory guarantees that elements of `$\Ext$' correspond to
*-monomorphisms of the C*-algebra to the Calkin algebra (see
\eqref{bdf0}, \eqref{bdf1} and \eqref{bdf2}). Further, on recalling
the algebra $B_W$ in \eqref{bwalgebra} with respect to $W \in
\Gr(p,A)$, the set of monomorphisms of the Burchnall-Chaundy algebra
into the algebras $B_W \subset\mathcal{L}(H_{\A})$ is indeed the
Jacobian $J(X)$, as proved by Burchnall and Chaundy and formalized
in \cite{Krich1,SW}. It is from this point of view that $J(X)$ can be
mapped to elements of `$\Ext$'. We make matters more precise by the
following:

\begin{theorem}\label{map}
There exists a well-defined map $\Phi: J(X) \lra \Ext(\uA)$ which
assigns to each morphism $\vp_W : \bA \lra B_W$, an extension of the
compact operators $\K(H_{\A})$ by $\uA$, with respect to
subspaces $W \in \Gr(p,A)$, and the subring inclusion
\eqref{inclusion2} induced via the integral operator $K$.
\end{theorem}
\begin{proof}
In order to construct $\Phi$ we first assign
to a point of $J(X)$, a morphism
$\vp_W : \bA \lra B_W$. Using the compact extensions map
$\ul{\varrho}_K$ in \eqref{bdf2} and the fact that $B_W \subset
\cL(H_{\A})$, the desired map is that for which the diagram
\begin{equation}
\xymatrix@C=2pc{\bA \ar[dr]_{\ul{\varrho}_K} \ar[r]^{\vp_W} & B_W
\subset \cL(H_{\A}) \ar[d]^{\Pi} \\ &\cL(H_{\A})/\K(H_{\A}) =
\Q(H_{\A})}
\end{equation}
commutes on restricting the projection $\Pi$ to $B_W$. Note that the
subring inclusion \eqref{inclusion2} is implicit in the definition
of $\ul{\varrho}_K$ in \eqref{bdf2}. More specifically, $\Phi: \bA
\lra \Ext(\uA)$ is determined via the assignment $\vp_W \mapsto
\ul{\varrho}_K$ for which $\ul{\varrho}_K = (\Pi\vert_{B_W}) \circ
\vp_W$ defines an extension of $\K(H_{\A})$ by $\uA$. ~$\Box$
\end{proof}

In diagram \eqref{chern-1} below, the vertical maps `$\ch_1$' denote
the (degree $1$) homology Chern character morphisms (see e.g.
\cite{Baum,Connesbook}).
\begin{theorem}\label{commutative}
The following diagram is commutative
\begin{equation}\label{chern-1}
\begin{CD}
K^1(C(X))  @> \hat{\mu}_{*}>> K^1(C(J(X)))
\\ @V \ch_1 VV   @VV \ch_1 V \\ H_1(X, \bQ) @>\mu_{*} \cong >>
H_1(J(X), \bQ)
\end{CD}
\end{equation}
Here the map $\mu_{*}$ is an isomorphism and the map $\hat{\mu}_{*}$
is injective. Equivalently, the map $\hat{\mu}_{*}: \Ext(X) \lra
\Ext(J(X))$, is injective. In particular, when the genus $g_X =1$,
the map $\hat{\mu}_{*}$ is an isomorphism.
\end{theorem}
\begin{proof}
Applying the functor $K^*$ to \eqref{exact1}, the resulting long
exact sequence yields an injective map $\hat{\mu}_{*}: K^1(C(X))
\lra K^1(C(J(X)))$. The commutativity of the diagram arises from
applying the homology Chern character homomorphism to each side.
Regarding the lower horizontal arrow $\mu_{*}$, we recall some
elementary facts concerning $X$ and its Jacobian $J(X)$ (see e.g.
\cite{GH}). Setting the genus $g_X= g$, if $\delta_1, \ldots,
\delta_{2g}$ are 1--cycles in $X$ forming a (canonical) basis for
$H_1(X, \bZ)$, then $H_1(X, \bZ) \cong \bZ\{\delta_1, \ldots,
\delta_{2g}\}$. Next, we identify $J(X) = \bC^{g}/\Lambda$ where
$\Lambda \subset \bC^{g}$ is a discrete lattice of maximal rank
$2g$. Accordingly, $H_1(J(X), \bZ) \cong \Lambda \cong
\bZ\{\lambda_1, \ldots, \lambda_{2g}\}$, for a basis $\lambda_1,
\ldots, \lambda_{2g}$ for $\Lambda$. Thus, from the (one-to-one)
assignment $\delta_i \mapsto \lambda_i,~1 \leq i \leq 2g$, it
follows that we have an isomorphism $H_1(X, \bZ) \cong H_1(J(X),
\bZ)$. Since both $X$ and $J(X)$ are quotients of their respective
covering spaces by torsion-free discrete groups, then on tensoring
the respective integral homology groups by $\bQ$, it also follows
that $H_1(X, \bQ) \cong H_1(J(X), \bQ)$. In the case of genus
$g_X=1$, $X$ is, up to the choice of a base point,  an elliptic
curve (a complex 1-dimensional torus), the map $\mu$ is an
isomorphism and consequently induces an isomorphism $\hat{\mu}_{*}:
K^1(C(X)) \lra K^1(C(J(X)))$, in other words, $\Ext(X) \cong
\Ext(J(X))$.
\end{proof}


\subsection{The flow of multiplication operators
and the $\tau$-function}\label{multflow}


The action on $\Gr(p,A)$ of the group of multiplication operators ${\Gamma}_{+}(\A)$
(see Appendix \ref{essential}) is induced via its
restriction to the subspace $\hp$ in any polarization. Here we
start by considering
certain elements of $\Gamma_{+}(\A)$ on which an operator-valued
$\tau$-function can be defined.


\begin{definition}
We say that $W \in \Gr(p,A)$ is \emph{transverse} if it is the graph
of a linear
operator $C: \hp \lra \hm$.
\end{definition}

Consider the element $q_{\zeta} \in {\Gamma}_{+}(\A)$ given by a
map $q_{\zeta}(z)= (1- z\zeta^{-1})$, for $\vert \zeta \vert > 1$,
whose inverse
is given by
\begin{equation}\label{q-inverse}
q_{\zeta}^{-1} = \{ \bmatrix a & b \\ 0 & d \endbmatrix : a
\in \Fred(\hp),~ d \in \Fred(\hm),~ b \in \K(H_{\A}) ~(\text{in fact},~ b
\in \cL_2(\hm,\hp)) \}.
\end{equation}
We define the \emph{operator-valued $\tau$-function}
\begin{equation}\label{tau-1}
\tau_W: {\Gamma}_{+}(\A) \lra \bC \otimes 1_{\A},
\end{equation}
relative to a transverse subspace $W \in \Gr(p,A)$ by
\begin{equation}\label{tau-2}
\tau_W(q_{\zeta}) = \det(1 + a^{-1}b C) \otimes {1}_{\A},
\end{equation}
where $C: \hp \lra \hm$ is the map whose graph is $W$.
The following lemma (and its proof) is essentially
that of \cite[Lemma 5.15]{SW}:
\begin{lemma}\label{unique-lemma}
Let $W \in \Gr(p,A)$ be transverse and let $f_0$ be the
unique element of $\hm$ such that
$1 + f_0 \in W$. Then for $\vert \zeta \vert > 1$, we have
$\tau_W(q_{\zeta}) = 1 + f_0(\zeta)$.
\end{lemma}
\begin{proof}
We start with the expression of \eqref{q-inverse} in which
$b: \hm \lra \hp$ takes $z^{-k}$ to
$\zeta^{-k}q_{\zeta}^{-1}$. Thus $a^{-1}b$ is a rank-1 map
that takes $f \in \hm$ to the constant
function $f(\zeta)$. The map $a^{-1}bC$ is also of rank-1 and
the infinite determinant
\begin{equation}
\tau_W(q_{\zeta}) = \det(1 + a^{-1}b C) \otimes {1}_{\A}
= (1 + \Tr(a^{-1}b C)) \otimes {1}_{\A}.
\end{equation}
Since $C$ maps $1$ to $f_0(z)$, the lemma follows.
\end{proof}
Now for $g \in {\Gamma}_{+}(\A)$, such that $g^{-1}W$ is transverse to $\hm$,
 the \emph{operator-valued Baker function} $\psi_W(g,\zeta)$
(see Appendix \ref{Baker}) is characterized as the unique function of the form
$1 + \sum_{i=1} a_i \zeta^{-i}$, whose boundary value $\vert
\zeta \vert \lra 1$, lies in the transverse
space $g^{-1}W$. On the other hand, we have as in \cite[\S3]{SW}
the relationship
\begin{equation}\label{tau-3}
\tau_W(g) = \sigma(g^{-1}W) (g^{-1} \sigma(W))^{-1},
\end{equation}
where $\sigma$ is a global section of the determinant
line bundle $\Det_{\Gr} \lra \Gr(p,A)$.
But then it is straightforward to see from \eqref{tau-3} that
\begin{equation}\label{tau-4}
\tau_W( g \cdot q_{\zeta}) (\tau_W(g))^{-1} = \tau_{g^{-1}W}
(q_{\zeta}) = \psi_{W}(g, \zeta),
\end{equation}
by the definition of $\psi_W$ and Lemma \ref{unique-lemma}. Thus \eqref{tau-4}
relates the $\tau$ and Baker functions in this operator-valued setting.
\begin{remark}
For $t=(t_1, \ldots, t_n)$, \eqref{tau-4} can be expanded to give (cf.
\cite{Alvarez,Sato, SW}):
\begin{equation}
\begin{aligned}
\psi_W(\zeta,t) &=  \tau(t - [\zeta^{-1}]) (\tau(t))^{-1} \\ &=
\tau(t_1 - 1/\zeta, t_2 - 1/2\zeta^2, t_3 - 1/3\zeta^3,
\ldots) (\tau(t_1, t_2, t_3, \ldots ))^{-1}.
\end{aligned}
\end{equation}
\end{remark}

Since applying the BDF theorem provides for each
element of the Burchnall-Chaundy C*-algebra $\uA$
an extension of the compact operators $\K(H_{\A})$,
we summarize matters by the following:
\begin{theorem}\label{tau-ext-theorem}
Each element $\esf \in \Ext(\uA)$ determines a
corresponding $\tau$-function $(\tau_W)_{\esf}$ relative
to a transverse subspace $W \in \Gr(p,A)$. In other words,
$\Ext(\uA)$ parametrizes a family of
$\tau$-functions $\{(\tau_W)_{\esf} \}_{\esf \in \Ext(\uA)}$
for such subspaces.
\end{theorem}
\begin{proof}
Observe that by considering the space $J_{\A}(X)$
 of monomorphisms $\bA \otimes \A \lra B_W$
(see Appendix \S\ref{BCring+}), we
obtain from the same principles of
Theorem \ref{map}, a map $\Phi_{\A}: J_{\A}(X) \lra \Ext(\uA)$.
Then on recalling the map
$\Xi: {\Gamma}_{+}(\A) \lra J_{\A}(X)$ in
\eqref{j-map2}, we have a well-defined map
\begin{equation}\label{Psi-map}
\Upsilon_{\A}: {\Gamma}_{+}(\A)  \lra \Ext(\uA),
\end{equation}
given via composition $\Upsilon_{\A} = \Phi_{\A} \circ \Xi$. The
result follows by
essentially finding an inverse for this map. Effectively, for each
 $\esf \in \Ext(\uA)$ providing an
extension of $\K(H_{\A})$, we have from
\eqref{q-inverse} a corresponding element of $\Gamma_{+}(\A)$ given by
\begin{equation}\label{q-inverse-e}
(q_{\zeta}^{-1})_{\esf} = \bmatrix a & b_{\esf} \\ 0 & d \endbmatrix .
\end{equation}
Hence, from \eqref{tau-2}, we obtain an element $(q_{\zeta})_{\esf}
 \in \Gamma_{+}(\A)$ to which there
corresponds a $\tau$-function
\begin{equation}\label{tau-2-e}
\tau_W((q_{\zeta})_{\esf}) = \det(1 + a^{-1}b_{\esf} C) \otimes {1}_{\A}
 = (1 + \Tr(a^{-1}b_{\esf} C)) \otimes {1}_{\A},
\end{equation}
associated to a transverse subspace $W \in \Gr(p,A)$.
\end{proof}


\appendix

\section{APPENDIX}\label{banachable}



\subsection{Hilbert modules}\label{semigroup}

Take $H$ to be a separable  Hilbert space.
Given a unital separable C*-algebra $\A$ one may consider the
standard (free countable dimensional) Hilbert module $H_{\A}$ over
$\A$,
\begin{equation}
H_{\A} = \{ \{\zeta_i\},~ \zeta_i \in \A~, ~i \geq 1:
\sum^{\infty}_{i = 1} \zeta_i \zeta_i^* \in \A ~\} \cong \oplus \A_i,
\end{equation}
where each $\A_i$ represents a copy of $\A$. We can form the
algebraic tensor product $H\otimes_{\alg} \A$ on which there is an
$\A$-valued inner product
\begin{equation}\label{inner1}
\langle x \otimes \zeta , y \otimes \eta \rangle = \langle x, y
\rangle \zeta^* \eta, \qquad x, y \in H,~ \zeta, \eta \in \A.
\end{equation}
Thus $H \otimes_{\alg} \A$ becomes an inner-product $\A$-module
whose completion is denoted by $H \otimes \A$. Given an orthonormal
basis for $H$, we have the following identification (a unitary
equivalence) given by $H \otimes \A \approx H_{\A}$ (see e.g.
\cite{Lance,DF}). For properties of compact and Fredholm operators
over Hilbert modules, see e.g. \cite{Bl}. We take $\cL(H_{\A})$ to
denote the C*-algebra of adjointable linear operators on $H_{\A}$
and $\Fred(H_{\A})$ to denote the space of Fredholm operators (see
e.g. \cite{Bl,Kasparov}). We  also make use of the Schatten ideals of
operators: the Banach spaces $\cL_{\ell} (H_{\A})$ are defined as
the subspaces of $\cL(H_{\A})$ consisting of operators $S$ with norm
satisfying $\Vert S \Vert_{\ell}^{\ell} = \Tr(S^*S)^{\ell/2} <
\infty$, (for $0\leq \ell \leq \infty$, where $\cL_{\infty}(H_{\A})
= \K(H_{\A})$ is the compact operators) \cite{Smith}.
 Since this will turn out to be an
essential ingredient (for example, in the Gelfand transform
\S\ref{commcase}) it is assumed that $H_{\A}$ is a
separable Hilbert *-module.

\begin{remark}\label{hilbertremark}
If $\phi$ denotes a state of $\A$, then we can produce a positive
semi-definite pre-Hilbert space structure on $H_{\A}$ via an induced
inner product $\langle v \vert w \rangle_{\phi} = \phi(\langle v
\vert w \rangle_{\A})$. From the Gelfand-Neumark-Segal (GNS) Theorem
there is an associated Hilbert space $H_{\phi}$ for which $\ell^2
\otimes H_{\phi}$ is the completion of $H_{\A}$ under this induced
inner product (see e.g. \cite{Bl,Fill}). When $\phi$ is understood
we shall simply denote the latter by $H(\A)$. Observe that $H(\A)$
contains $H_{\A}$ as a dense vector subspace which is a right
$\A$-module, if $\phi$ is completely positive.
The assignment of Hilbert spaces $H \mapsto H(\A)$ and
use of the $\A$-valued inner product in \eqref{inner1}, thus allows
us to modify various results  that were originally established for
the various objects $(H, \langle~,~\rangle)$ (such as the bounded
linear operators, the compact operators, etc. in the case $\A=
\bC$).
\end{remark}


\subsection{On the geometry of the Grassmannian $\Gr(p,A)$ and related properties}\label{projections}

Let $A$ be a unital complex
Banach(able) algebra with group of units $G(A)$ and space of
idempotents $P(A)$; for basic facts about Banach algebras we refer
to \cite{RGD,Rickart}. For $p \in P(A)$, we say that the orbit of
$p$ under the inner automorphic action of $G(A)$ is \emph{the
similarity class of} $p$ and denote the latter by $\Lambda =\Sim (p,A) =
G(A)*p$. Following \cite{DG2} there exists a space $V(p,A)$ modeled
on the space of proper partial isomorphisms of $A$ which is the
total space of a locally trivial principal $G(pAp)$--bundle
\begin{equation}\label{stiefel1}
G(pAp) \hookrightarrow V(p,A) \lra \Gr(p,A),
\end{equation}
where the base is the Grassmannian $\Gr(p,A)$ viewed as the image of
$\Lambda =\Sim (p,A)$ with respect to the projection map. The space
$\Gr(p,A)$ is a (complex) Banach manifold and the construction of
the bundle in \eqref{stiefel1} generalizes the usual Stiefel bundle
construction in finite dimensions. If $A[p]$ denotes the
commutant of $p \in A$, we set $G[p] = G (A[p])$.  If we further assume that $A$ has
the unital associative *-algebra property, there is the
corresponding unitary group $U(A) \subset G(A)$. Thus on setting $U[p] = U(A) \cap G[p]$,
we can specialize \eqref{stiefel1} as an analytic, principal
$U[p]$-bundle $U[p] \hookrightarrow V(p,A) \lra \Gr(p,A) \cong U(A)/U[p]$.

\begin{remark}\label{others-remark}
The references \cite{DG2,DGP3} describe the possible analytic structures of the spaces $P(A), \Gr(p,A), V(p,A)$ and $\Lambda$,
and related objects in some depth. For such spaces we wish to point out that in the Banach manifold and C*-algebra setting there is available
related work by other authors, using different techniques, that is applicable to a broad scope of independent problems to those considered here; see for instance
\cite{ALRS,Beltita1,Beltita2,Beltita3,CPR1,Helm1,MS1,MS2,PR1,Raeburn,Upm,Wilkins} (and the references
to related work cited therein). The present authors thus acknowledge an inevitable overlap with certain technical details in this regard.
\end{remark}

We proceed with $A$ as defined in \eqref{staralgebra}
and take $E$ to be its underlying complex Banach space that admits a splitting subspace decomposition
compatible with the polarization \eqref{polar1} (for the
theory of compact and Fredholm operators extended to Banach spaces
we refer to \cite{ZKKP}). In the following we restrict to the $\ell =2$ (Hilbert-Schmidt) case
and consider as in \cite{DGP1,PS,SW}
the explicit form of ${G}(A)$ as the group of automorphisms
of $\Gr(p,A)$ as given by
\begin{equation}\label{ghat1}
\begin{aligned}
G(A) = &\{ \bmatrix T_1 & S_1 \\ S_2& T_2 \endbmatrix : T_1
\in \Fred(\hp),~ T_2 \in \Fred(\hm),~ S_1 \in \cL_2(\hm,\hp), \\
~ &S_2 \in \cL_2(\hp,\hm)\},
\end{aligned}
\end{equation}
whereby $\Gr(p, A)= {G}(A)/G[p]$ (see
\cite{DG2,DGP1}). Note that for $p \in P^{\sa}$, we have $U[p] \cong U(A) \cap G[p] \cong U(A) \cap (U(\hp) \times U(\hm))$.
If we take $\rho$ to denote the left action of
$U[p]$ on $H(p)$ where the latter denotes the underlying Hilbert
space of $pA \cong \hp$, then following \cite[\S 5]{DG3} we have the associated universal vector bundle to \eqref{stiefel1}
\begin{equation}\label{unibundle4}
\gamma(p,A) = U(A) \times_{U[p]} H(p) \lra U (A) /U[p] = \Gr(p,A),
\end{equation}
as was established in \cite[\S 5]{DG3}.
We remark that homogeneous vector bundles (such as \eqref{unibundle4}) arising
from certain representations have been more recently studied in \cite{Beltita3}.
From works such as \cite{AL,Beltita2,Beltita3,Helm1,PS} it is straightforward to deduce that
$\Gr(p,A)$ admits the structure of a Hermitian
manifold and that the universal vector bundle  ${\E}(p,A) \lra \Gr(p,A)$ admits the
structure of a Hermitian vector bundle.

The space $\Gr(p,A)$ may be realized more specifically in the
following way. Suppose that a fixed $p \in P(A)$ acts as the
projection of $H_{\A}$ on $\hp$ along $\hm$. Then in a similar way to \cite{PS,SW}, the space
$\Gr(p,A)$ is the Grassmannian consisting of subspaces $W = r(H_{\A})$, for $r \in
P(A)$, such that:
\begin{itemize}\label{proj2}
\item[(1)]
the projection $p_{+} = pr~:~ W \lra \hp$ is in $\Fred(H_{\A})$, and

\med
\item[(2)]
the projection $p_{-} = (1 - p)r ~:~ W \lra \hm$ is in $\cL_2(\hp,
\hm)$.
\end{itemize}
Alternatively, for (2) we may take projections $q \in P(A)$ such
that for the fixed $p \in P(A)$, the difference $q - p \in
\cL_2(\hp, \hm)$. Further, there is the \emph{big cell} $C_{\rm{b}}
= C_{\rm{b}}(p_1,A) \subset \Gr(p,A)$ defined as the collection of all
subspaces $W \in \Gr(p,A)$ such that the projection $p_{+} \in
\Fred(H_{\A})$ is an isomorphism.

In regard of *-subalgebras of $A$, the following lemma is easily deduced from \cite[\S2.4]{DGP1}
\begin{lemma}
Let $B \subset A$ be a Banach *-subalgebra of $A$, with inclusion
$h: B \lra A$. Then there is an induced inclusion of Grassmannians
$\Gr(q,B) \subset \Gr(p,A)$ where for $q \in P(B)$ we have set $p =
h(q) \in P(A)$.
\end{lemma}

\begin{example}\label{restrictexample}
Let $\B \subset \A$ be a C*-subalgebra. Then it is straightforward to see
that we have an inclusion
$\cL_J(H_{\B}) \lra A= \cL_J(H_{\A})$.
In particular, when $\B \cong \bC$ and $H = L^2(S^1, \bC)$ for
which there is a polarization $H=H_{+} \oplus H_{-}$ ($H_{+} \cap
H_{-} = \{0\}$), the inclusion $\cL_J(H) \lra \cL_J(H)\otimes \A$
induces an inclusion
$\Gr(H_{+}, H) \subset \Gr(p,A)$ where $\Gr(H_{+}, H)$ is the
`restricted' Grassmannian as used in \cite{PS,SW}.
\end{example}

\subsection{The Burchnall-Chaundy ring
and holomorphic geometry}\label{essential}

Let us briefly recall how for $\A= \bC$ the relevant algebra of differential
operators of the KP flows, is related to
the Grassmannian $\Gr(q,B)$ \cite{SW}.
Let $\bB$ denote the algebra of analytic functions $U \lra \bC$
where $U$ is a connected open neighborhood of the origin in $\mathbb
C$. The (noncommutative) algebra $\bB[\del]$ of linear differential
operators with coefficients in $\bB$, consists of expressions
\begin{equation}\label{ODO}
\sum_{i = 0}^N a_i~ \del^i,~(a_i \in \bB,~~{\text{for some}}~ N \in
\bZ ).
\end{equation}
Here $\del := \del/\del x$ and the $a_i$ can be regarded as
 operators on functions, with multiplication
\begin{equation}
[\del, a] = \del a - a \del = \del a/\del x.
\end{equation}
More generally, we pass to the algebra $\bB[[\dm]]$ of formal
pseudodifferential operators with coefficients in $\bB$. This
algebra is obtained  by formally inverting the operator $\del$ (see
e.g. \cite{Sato}) and taking Laurent series as in (\ref{ODO}), with
$-\infty <i\le N$.

The ring $\bA$ is assumed to be a commutative subring of
$\bB[\partial ]$, whose joint spectrum is a complex irreducible
curve $X' = \Spec(\bA)$ with completion a non-singular algebraic
curve $X = X'\cup \{x_{\infty}\}$ of genus $g_X \geq 1$. We recall
from \cite{SW} the following associated quintuple of data $(X,
x_{\infty}, z,\cL, \vp)$: $\cL \lra X$ is a holomorphic line bundle,
$x_{\infty}$ is a smooth point of $X$, $z$ the inverse of a local
parameter on $X$ at $x_{\infty}$, where $z$ is used to identify a
closed neighborhood $X_{\infty}$ of $x_{\infty}$ in $X$ with a
neighborhood of the disk $D_{\infty} = \{z : \vert z \vert \geq 1\}$
in the Riemann sphere, and $\vp$ is a holomorphic trivialization of
$\cL$ over $X_{\infty} \subset X$. Subsequently, we consider
$L^2$--boundary values of $\cL$ over $X\backslash {D_{\infty}}$ and
$\vp$ identifies sections of $\cL$ over $S^1$ with $\bC$-valued
functions. Thus we arrive at the (separable) Hilbert space $H =
L^2(S^1, \bC)$ together with a
polarization $H = H_{+} \oplus H_{-}$ (with $H_{+} \cap H_{-} = \{0\}$)
in the case $B = \mathcal L_J (H)$. Further, the quintuple $(X,
x_{\infty}, z,\cL,\vp)$ can be mapped, by the \textit{Krichever
correspondence}, to a point $ W \in \Gr(q,B)$ (cf \cite{Krich1,SW}).
Following \cite{SW}, the properties of the projections $p_{\pm}$ in
\S\ref{projections} apply and the kernel (respectively, cokernel) of
the projection $p_{+} : W \lra H_{+}$ is isomorphic to the
sheaf-cohomology group $H^0(X, \cL_{\infty})$ (respectively, $H^1(X,
\cL_{\infty})$) where $\cL_{\infty} = \cL \otimes \mathcal
O(-x_{\infty})$.

Let $\Gamma_{+}$ denote the group of real analytic
mappings $f: S^1 \lra \bC^*$, extending
to holomorphic functions $f: D_0 \lra \bC^*$ in the closed unit disk
$D_0 =\{z \in \bC: \vert z \vert \leq 1 \}$, such that $f(0)=1$.
The natural action of $\Gamma_{+}$ on $\Gr(q,H)$ corresponds to its action
on $(X,x_{\infty}, z,\cL,\vp)$, and further from \cite{SW}(Prop. 6.9)
there exists a surjective
homomorphism
\begin{equation}\label{j-map1}
\Gamma_{+} \lra J(X),
\end{equation}
where $J(X)$ is the Jacobian torus of $X$
(recall that $J(X)$ is the commutative group of holomorphic line
 bundles of zero degree over
$X$). The  Krichever correspondence \cite{Krich1} (see also
\cite[Proposition 6.2]{SW}) links the data of the quintuple to a
flow of multiplication operators $\Gamma_{+}$ on $H$, thus inducing
 a linear flow on
$J(X)$ of $X$. Furthermore, this flow corresponds to the evolution
of solutions of the generalized KdV flow.

We now take the maps in the definition of $\Gamma_{+}$ to be $\A$-valued
and denote the resulting group by $\Gamma_{+}(\A)$.
An action of ${\Gamma}_{+}(\A) \subset G(A)$ on $\Gr(p,A)$ is induced via its
restriction to the subspace $\hp$ in any polarization.
Recalling that
$\A = C(Y)$, then clearly, in the special case $Y=\{pt\}$ we recover
the above group $\Gamma_{+}$. In view of \eqref{aform} and
$\Gr(q,B) = \Gr(H_{+},H)$, the restriction
${\Gamma}_{+}(\A)\vert\Gr(q,B)$ as a $Y$-parametrized family
$\{(\Gamma_{+})_y\}_{y \in Y}$. Effectively, we then have a flow of
multiplication operators as given by
\begin{equation}\label{flowops}
\begin{aligned}
\Gamma_{+}(\A) &= \{ (\exp(\sum_{\alpha} a_{\alpha} z^{a}):
a_{\alpha} \in \A \} \\
&\cong \{ (\exp(\sum_{a} t_a z^{a}))_y \}_{y \in Y}.
\end{aligned}
\end{equation}


We also have the group $\Gamma_{-}(\A)$ of
multiplication operators on $H_{\A}$ given by the group of continuous maps $g:
\bC\backslash \Int(D) \lra \A$ such that $g$ is real analytic in
$x \in \bC \backslash \Int(D)$ extending to $g(z)$ holomorphic in
$z \in \bC \backslash \Int(D)$, and $g(\infty) = 1$.

\subsection{The Burchnall-Chaundy ring with coefficients in $\A$}
\label{BCring+}

Continuing with $\A\cong C(Y)$ a commutative (separable) C*-algebra,
we now consider the algebra $\mathcal{O}(U, \A)$ of $\A$-valued
analytic functions
$U \lra \A$ where $U$ is a connected open neighborhood of the origin
in $\bC$. In turn, we consider the algebra $\mathcal{O}(Y,\bB[\del])$
of linear
differential operators with coefficients in $\mathcal{O}(Y,\bB)$
consisting of
expressions (\ref{ODO}) where essentially the same discussion in \S
\ref{essential} applies verbatim (note that this
algebra contains the algebraic
tensor product $\bB[\del] \otimes \A$). In particular,
 the coefficients $a_i$ are now thought of as $\A$-valued
functions.

Relative to $W \in \Gr (p, A)$, the algebra
\begin{equation}\label{bwalgebra}
B_W = \{ f(z) = \sum^N_{s = - \infty} c_s z^{s} : s \in \mathbb N,~
c_s \in \A,~ f(z)W\subset W \},
\end{equation}
(denoted by $A_W$ in \cite{SW}) contains the coordinate
ring of the curve $X \smallsetminus \{x_{\infty}\}$. In fact,
the set of monomorphisms $\bA \lra B_W$ is identified with
$J(X)$ \cite{Krich1,SW}. Let then $J_{\A}(X)$ denote the set of
monomorphisms $\bA \otimes \A \lra B_W$. In view of \eqref{j-map1}
and the appropriate inclusions, we
obtain a commutative diagram
\begin{equation}\label{jmap-comm}
\begin{CD}
\Gamma_{+}  @>>> \Gamma_{+}(\A)
\\ @VVV   @VVV \\ J(X) @>>> J_{\A}(X)
\end{CD}
\end{equation}
and thus from \eqref{j-map1} a well-defined surjective homomorphism
\begin{equation}\label{j-map2}
\Xi: \Gamma_{+}(\A) \lra J_{\A}(X).
\end{equation}
Observe that since $\A = C(Y)$, \eqref{j-map2} is equivalent to
a family of maps $Y \lra \{\Gamma_{+} \lra J(X) \}$.

\subsection{The abstract operator-valued Baker function}\label{Baker}

Here we consider subspaces $W \in \Gr(p, A)$ of the form $W = gh_g \hp$
with $g \in {\Gamma}_{+}(\A)$ and $h_g \in {\Gamma}_{-}(\A)$.
Let ${\Gamma}_{+}^W (\A)$ denote the subgroup
of elements $g \in {\Gamma}_{+}(\A)$ such that $g^{-1}W$
is transverse to $\hm$. Also, for
$g \in {\Gamma}_{+}^W(\A)$, we consider orthogonal projections
\begin{equation}
p_{1}^g = p\vert g^{-1}(W) :~g^{-1}(W) \lra \hp ,~ p_{1}^g \in \Fred(H_{\A}),
\end{equation}
(i.e. $p_{1}^g \in \Fred(H_{\A}\cap P^{\sa}(A)$) and define in relationship to the
big cell $C_b$, the subgroup of
${\Gamma}_{+}^W(\A)$ as given by
\begin{equation}
\wti{\Gamma}_{+}^W (\A) = \{ g \in {\Gamma}_{+}(\A) :  p_{1}^g ~\text{is
an isomorphism}  \}.
\end{equation}

\begin{definition}\label{opbaker}
{The operator-valued Baker function} $\psi_W$ associated to the
subspace $W \in C_{\mathrm{b}}\subset \Gr(p,A)$
is defined formally as:
\begin{equation}\label{Baker1}
\psi_W = (p_{1}^g)^{-1}(\mathbf 1) = (\sum_{s =0}^{\infty} a_s (g)~
\zeta^{-s})~ g(\zeta) \in W g(\zeta) ,
\end{equation}
where $g \in \wti{\Gamma}_{+}^W(\A)$ and the $a_s$ are analytic
$\A$-valued (operator-matrix) functions on ${\Gamma}_{+}^W(\A)$
extending to all $\A$-valued functions $g \in {\Gamma}_{+}(\A)$
meromorphic in $z$ (cf. \cite{DGP1}).
\end{definition}

\subsection{A formal integral operator}\label{integral-op}

As in \cite[\S6]{DGP1}, following \cite{SW}, there
exists a formal integral operator $K \in \bB[[\dm]]$ given by
\begin{equation}\label{integralop}
K = 1 + \sum^{\infty}_{s=1} a_s(x)~ \del^{-s},
\end{equation}
(where the $a_s$ are $\A$-valued functions) unique up to a constant
coefficient operator such that $L = K (\del^n) K^{-1}$ belongs to
$\bA$. Under the above correspondence the (formal) Baker function
$\psi_W$ is defined as $\psi_W = K e^{xz}$, the main point being
that the function $\psi_W$ will be an eigenfunction for the operator
$L^{1/n}=\partial +[\hbox{lower-order terms]}$, that is, $L^{1/n} \psi_W =  z~ \psi_W$,
and accordingly
\begin{equation}
\begin{aligned}
\psi_W (x , z) = (1+ \sum_{s =1}^{\infty} a_s (x)~ z^{-s}) ~e^{x z}.
\end{aligned}
\end{equation}
Using a form of the Sato correspondence \cite{Sato}, we established
(for $\A$ not necessarily commutative):

\begin{theorem}{\rm{\cite[Theorem 6.1]{DGP1}}}\label{rep1}
Given the Baker function $\psi_W$ associated to a subspace $W \in
\Gr(p,A)$, the Burchnall-Chaundy ring $\bA \subset
 \bB[[\dm]]\otimes \A$ is conjugated into
$A\subset \mathcal L_J(H_{\A})$ as a commutative subring, the
conjugating integral operator $K$ being unique up to constant
coefficient operators.
\end{theorem}

\subsection{The $Y$-parametrized holomorphic data}\label{ydata}

Since we have effectively tensored the coefficients of $\bA$ by
$\A= C(Y)$, we can modify the discussion in \S \ref{essential}
using the expression for $A$ in \eqref{aform}.
Specifically, the same construction involving the data in \S
\ref{essential} yields a \emph{$Y$-parametrized map} to $\Gr(q,B) =
\Gr(H_{+},H)$ where we recall $B= \cL_J(H)$ from \S \ref{commcase}.
Consequently, for a $Y$-parametrized quintuple in \S\ref{essential},
$y \in Y$, we obtain the assignment
\begin{equation}\label{quin}
\{(X, x_{\infty}, z_y, \cL_y, \vp_y)\}_{y\in Y} ~\mapsto ~\text{certain
points}~ W_y \in Y \times \Gr(q,B).
\end{equation}

As was noted in \S\ref{multflow}, the restriction
${\Gamma}_{+}(\A)\vert\Gr(q,B)$ acts as a family of multiplication
operators $\{(\Gamma_{+})_y\}_{y \in Y}$ on subspaces $W \in
\Gr(q,B)$. Following \cite[\S6]{SW}, an element $g \in \Gamma_{+}$
serves as a transition function for a  line bundle over $X$ (that
is, $g \in \Gamma_{+}$ determines a point in the Picard group
${\mathrm{Pic}}^{g_X}(X)$ \cite{GH} which is then twisted into a
point of the Jacobian  $J(X)$; we recall that $g_X$ denotes the
genus of $X$). The restricted action ${\Gamma}_{+}(\A)\vert\Gr(q,B)$
gives a parametrized version:

\begin{proposition}\label{jacobian}
Let $\mathcal J_0(X \times Y)$ denote the space of topologically
trivial line bundles on $X \times Y$. Then there exists a
well-defined map
\begin{equation}
{\Gamma}_{+}(\A)\vert\Gr(q,B) \lra \mathcal J_0(X \times Y),
\end{equation}
given by $g_y \mapsto \cL_{g_y}$, and an induced action of
${\Gamma}_{+}(\A)\vert\Gr(q,B)$ on $\mathcal J_0(X \times Y)$.
\end{proposition}

\begin{proof}
We follow \cite[Proposition 6.9]{SW} closely. Let $X_0 = X\backslash
X_{\infty}$, and let $\mathcal L_{g_y} \lra X \times Y$ be the  line
bundle obtained by taking topologically trivial line bundles over
$X_0 \times Y$ and $X_{\infty} \times Y$ and glueing them by
$g_y=(g,y)$ over an open neighborhood of $S^1 \times Y$, where $g
\in \Gamma_{+}$. This line bundle has degree $g_X$, so it is not
topologically trivial, but by changing $g_y$ by an element of
$\Gamma_-$ we achieve degree zero. Thus we obtain a map
\begin{equation}
\begin{aligned}
{\Gamma}_{+}(\A)\vert\Gr(q,B)  &\lra \mathcal J_0(X \times Y), \\
g_y &\mapsto \cL_{g_y}
\end{aligned}
\end{equation}
where $\cL_{g_y}$ has a $\vp_g$-induced trivialization
\begin{equation}
\vp_{g_y} = (\vp_g,y) : \cL_{g_y}\vert X_{\infty} \times Y \lra \bC
\times X_{\infty} \times Y.
\end{equation}
Consequently, $(\Gamma_{+})_y$ acts on $(X,x_{\infty},z)$ and on
$(\cL, \vp)$ via the tensor product with $(\cL_{g_y}, \vp_{g_y})$.
\end{proof}

In this way the action of ${\Gamma}_{+}(\A)\vert\Gr(q,B)$ on
$Y$-parametrized solutions of the generalized $n$-th KdV equations
corresponds to $\cL \lra \cL\otimes\cL_{g_y}$. For a fixed $y \in
Y$, the assignment $g \mapsto \cL_g$ defines a surjective group
homomorphism $\Gamma_{+} \lra J(X)$, as in the case $\A \cong \bC$
\cite[Remark 6.8]{SW}.




\begin{thebibliography}{99}

\bibitem{Alvarez}
A. \'Alvarez V\'asquez, J.M. Mu\~noz Porras and F.J. Plaza Mart\'in,
The algebraic formalism of soliton equations over arbitrary base
fields. In `Workshop on Abelian varieties and Theta Functions'
(Morelia MX, 1996), 3--40, \emph{Aportaciones Mat. Investig.} \textbf{13},
 Soc.
Mat. Mexicana, Mexico 1998.

\bibitem{ALRS}
E. Andruchow, A. Larotonda, L. Recht and D. Stojanoff,
Infinite dimensional homogeneous reductive spaces and finite
index conditional expectations, \emph{Illinois J. Math.} \textbf{41}
(1997), 54--76.

\bibitem{AL}
E. Andruchow, A. Larotonda,
Hopf-Rinow theorem in the Sato Grassmannian, \emph{J. Funct. Anal.} \textbf{225}
(2008) (7), 1692--1712.

\bibitem{Ball2}
J. Ball and V. Vinnikov, Lax-Phillips scattering and conservative
linear systems: a Cuntz-algebra multidimensional setting,
\emph{Mem. Amer. Math. Soc.} \textbf{178} (2005), no. 837, iv+101
pp.

\bibitem{Baum}
P. Baum and R. G. Douglas, K--homology and index theory,
\emph{Operator algebras and applications, Part 1, Proc. Symp. Pure
Math.} \textbf{38}, Amer. Math. Soc, Providence RI, 1982, 117---173.

\bibitem{Beltita1}
D. Belti\c{t}\u{a}, \emph{Smooth Homogeneous Structures in Operator
Theory}, Monographs and Surveys in Pure and Appl. Math.
\textbf{137}, Chapman and Hall/CRC, Boca Raton FL, 2006.

\bibitem{Beltita2}
D. Belti\c{t}\u{a} and J. Gal\'e, Holomorphic geometric models for
representations of C*--algebras, \emph{J. Functional Anal.} \textbf{255} (10) (2008) , 2888--2932.

\bibitem{Beltita3}
D. Belti\c{t}\u{a} and J. Gal\'e, On complex infinite-dimensional
Grassmann manifolds, \emph{Complex Anal. Oper. Theory} \textbf{3} (2009),
 739--758.

\bibitem{Bl}
B. Blackadar, \emph{K-Theory for Operator Algebras}, Springer
Verlag, 1986.

\bibitem{BDF}
L. G. Brown, R. G. Douglas and P. A. Fillmore, Extensions of
C*-algebras and K-homology, \emph{Ann. of Math.} \textbf{105} (1977),
265--324.

\bibitem{Connesbook}
A. Connes, \emph{Noncommutative Geometry}, Academic Press, 1994.

\bibitem{CS}
A. Connes and G. Skandalis, The longitudinal index theorem for
foliations, \emph{Publ. Res. Inst. Math. Sci. Kyoto University}
\textbf{20} (6), (1984), 1139--1183.


\bibitem{CPR1}
G. Corach, H. Porta and L. Recht, Differential geometry of systems
of projections in Banach algebras, \emph{Pacific J. Math.}
\textbf{143}~(2), (1990), 209--228.

\bibitem{CD}
M. J. Cowen and R. G. Douglas, Complex geometry and operator theory,
\emph{Acta Math.} \textbf{141} (1978), 187--261.



\bibitem{Dix}
J. Dixmier, \emph{Les C*-alg\`{e}bres et leurs representation},
Gauthier-Villars, Paris, 1964.


\bibitem{RGD}
R. G. Douglas, \emph{Banach Algebra Techniques in Operator Theory}
(2nd Edition). Graduate Texts in Mathematics \textbf{179}, Springer
Verlag, New York-Berlin,  1998.

\bibitem{DF}
M. J. Dupr\'e and P. A. Fillmore, Triviality theorems for Hilbert
modules. In \emph{Topics in modern operator theory
 (Timi\c{s}oara/Herculane, 1980)}, pp. 71--79
\emph{Operator theory: Advances and Applications} \textbf{2},
 Birkh\"auser Verlag, Boston, Mass, 1981.


\bibitem{DG2}
M. J. Dupr\'e and J. F. Glazebrook,  The Stiefel bundle of a Banach
algebra, \emph{Integral Equations Operator Theory} \textbf{41} No. 3,
(2001), 264--287.

\bibitem{DG3}
M. J. Dupr\'e and J. F. Glazebrook, Holomorphic framings for
projections in a Banach algebra, \emph{Georgian Math. J.} \textbf (2002)
No. 3, 481--494.

\bibitem{DGP1}
M. J. Dupr\'e, J. F. Glazebrook and E. Previato, A Banach algebra
version of the Sato Grassmannian and commutative rings of
differential operators, \emph{Acta Applicandae Math.} \textbf{92} (3)
(2006), 241--267.

\bibitem{DGP3}
M. J. Dupr\'e, J. F. Glazebrook and E. Previato,
Curvature of universal bundles of Banach algebras. In,
`Proceedings of the
International Workshop on Operator Theory and Applications-2008',
eds. J. Ball {et al.},
\emph{Operator Theory: Advances and Applications} \textbf{202} (2009),
 195--222.

\bibitem{DGP4}
M. J. Dupr\'{e}, J. F. Glazebrook and E. Previato: Differential
algebras with Banach-algebra coefficients II:
The operator cross-ratio Tau function and the Schwarzian derivative,
(submitted for publication).


\bibitem{Fill}
P. A. Fillmore, \emph{A User's Guide to Operator Algebras},
Wiley--Interscience, New York, 1996.

\bibitem{GH}
P. A. Griffiths and J. Harris, \emph{Principles of Algebraic
Geometry}, J. Wiley \& Sons, New York, 1978.

\bibitem{Helm1}
G. F. Helminck and A. G. Helminck, Hilbert flag varieties and their
K\"ahler structure, \emph{J. Phys. A: Math. Gen.} \textbf{35} (2002),
8531--8550.

\bibitem{Helton}
J. W. Helton, \emph{Operator Theory, Analytic Functions,
 Matrices and Electrical Engineering}, CBMS Reg. Conf. Ser. \textbf {68},
 Amer. Math. Soc., Providence RI, 1987.

\bibitem{Kasparov}
G. G. Kasparov, The operator K-functor and extensions of
C*-algebras, \emph{Math. U.S.S.R. Izv.} \textbf{16} (1981) No 3513-572.
Translated from \emph{Izv. Akad. Nauk. S.S.S.R. Ser. mat}
\textbf{44} (1980), 571--636.

\bibitem{Krich1}
I. M. Krichever, Integration of nonlinear equations by the methods
of algebraic geometry, \emph{Funkcional. Anal. i Prilo\v{z}en.}
\textbf{11}  No. 1 (1977), 15--31.


\bibitem{Lance}
E. C. Lance, \emph{Hilbert C*-Modules}, {London Math. Soc. Lect.
Notes} \textbf{210}, Cambridge Univ. Press, 1995.

\bibitem{MS1}
M. Martin and N. Salinas, Flag manifolds and the Cowen--Douglas theory,
\emph{Journal of Operator Theory} \textbf{39} (1997), 329--365.

\bibitem{MS2}
M. Martin and N. Salinas, Differential geometry of generalized
Grassmann manifolds in C*-algebras, \emph{Operator theory and boundary
eigenvalue problems
(Vienna, 1993)}, 206--243,
\emph{Operator Theory: Advances and Applications} \textbf{80},
Birh\"auser Verlag, Basel, 1995.




\bibitem{PR1}
H. Porta and L. Recht, Spaces of projections in a Banach algebra,
\emph{Acta Cient. Venez.} \textbf{38}, (1987), 408--426.

\bibitem{PS}
A. Pressley and G. Segal, \emph{Loop Groups and their
Representations}, Oxford Univ. Press, 1986.

\bibitem{Raeburn}
I. Raeburn, The relationship between a commutative Banach algebra and its maximal ideal space,
\emph{J. Funct. Anal.} \textbf{25} (1977), 366--390.

\bibitem{Rickart}
C. E. Rickart, \emph{General Theory of Banach Algebras}, Krieger
Publishing, Huntington, 1960.

\bibitem{Roe}
J. Roe, \emph{Elliptic operators, Topology and Asymptotic Methods},
Pitman Res. Notes in Math. \textbf{179}, Longman Scientific/Wiley,
New York 1988.

\bibitem{Sato}
M. Sato, The KP hierarchy and infinite dimensional Grassmann
manifolds, \textit{Theta functions -- Bowdoin 1987, Part 1}
(Brunswick, ME, 1987), 51--66, Proc. Sympos. Pure Math., 49, Part 1,
Amer. Math. Soc., Providence, RI, 1989.

\bibitem{SW}
G. Segal and G. Wilson, Loop groups and equations of KdV type,
\emph{Publ. Math. IHES }\textbf{61}, (1985), 5--65.

\bibitem{Sk}
G. Skandalis, Kasparov's bivariant K-Theory and applications, \emph{Expo.
Math.} \textbf{9} (1991), 193--250.

\bibitem{Smith}
J. F. Smith, The $p$-classes of a Hilbert module, \emph{Proc. Amer. Math.
Soc.} \textbf{36} (2) (1972), 428--434.


\bibitem{Upm}
H. Upmeier, \emph{Symmetric Banach Manifolds and Jordan
C*-Algebras}, Math. Studies \textbf{104},  North Holland,
Amsterdam-New York-Oxford, 1985.

\bibitem{Wells}
R. Wells, \emph{Differential Analysis on Complex Manifolds}, Grad.
Texts in Math. \textbf{65}, Springer-Verlag, New York Heidelberg
Berlin, 1980.

\bibitem{Wilkins}
D. R. Wilkins, The Grassmann manifold of a C*-algebra,
\emph{Proc. Royal Irish Acad.} \textbf{90} A No. 1 (1990), 99-116.

\bibitem{ZKKP}
M. G. Zaidenberg, S. G. Krein, P. A. Kushment and A. A. Pankov,
Banach bundles and linear operators, \emph{Russian Math. Surveys}
\textbf{30} No. 5, (1975), 115--175.

\end{thebibliography}
\end{document}